\documentclass[12pt]{amsart}
\usepackage{amsthm,amsfonts,amsmath,amssymb,amscd} 

\usepackage{graphicx}
\usepackage{xcolor}
\definecolor{light-gray}{gray}{0.95}

\newtheorem{thm}{Theorem}

\newtheorem{lm}{Lemma}
 
\theoremstyle{definition}
\newtheorem{crl}{Corollary}
\newtheorem{dfn}{Definition}

\theoremstyle{remark}

\newcommand{\des}{\mathop{\rm des}\nolimits}

\newcommand{\stir}[2]{S(#1, #2)}
\newcommand{\ctir}[2]{s(#1, #2)}
\newcommand{\sstir}[2]{\left \{ {#1} \atop {#2} \right \}}
\newcommand{\cstir}[2]{\left[ {#1} \atop {#2} \right]}
\newcommand{\GS}{\mathcal{SP}}

\newcommand{\stirr}[2]{S_{\bf 3}(#1, #2)}
\newcommand{\ctirr}[2]{s_{\bf 3}(#1, #2)}
\newcommand{\stirrr}[2]{{S}_k(#1, #2)}
\newcommand{\ctirrr}[2]{{s}_k(#1, #2)}
\newcommand{\stirg}[2]{S_{\mathbf{k}}(#1, #2)}
\newcommand{\ctirg}[2]{s_{\mathbf{k}}(#1, #2)}
\newcommand{\B}{b}
\newcommand{\G}{g}

\title{Stirling permutations on multisets}
\author{Askar Dzhumadil'daev, Damir Yeliussizov}
\date{}

\begin{document}

\maketitle

\begin{center}
Kazakh-British Technical University \\
59 Tole bi St, 050000, Almaty, Kazakhstan\\
\texttt{dzhuma@hotmail.com, yeldamir@gmail.com}
\end{center}

\begin{abstract}  A permutation $\sigma$ of a multiset  is called Stirling permutation if $\sigma(s)\ge \sigma(i)$ as soon as  $\sigma(i)=\sigma(j)$ and $i<s<j.$ 
In our paper we study Stirling polynomials that arise in the generating function for descent statistics on Stirling permutations of any multiset.  We develop generalizations of the classical Stirling numbers and present their combinatorial interpretations. Particularly, we apply the theory of $P$-partitions. Using certain specifications we also introduce the Stirling numbers of odd type and generalizations of the central factorial numbers.  
\end{abstract}

\section{Introduction}

Let ${\bf k} = (k_1, \ldots, k_n)$ and 
${\bf n}=\{1^{k_1}, \ldots, n^{k_n} \}$  
be a multiset of type ${\bf k},$ i.e., $k_i$ is a number of copies of the element $i.$ A permutation of a multiset is a sequence of  its elements. We say that the permutation
$\sigma$ of a multiset is a {\it Stirling permutation} if $\sigma(s) \ge \sigma(i)$ as soon as $\sigma(i)=\sigma(j)$ and $i<s<j.$ Stirling permutations were introduced by Gessel and Stanley \cite{stirpol} in case of the multiset $\{1^2, \ldots, n^2 \}$. 

Denote by $\GS_{\bf k}$ the set of Stirling permutations of ${\bf n}.$  
For $\sigma\in \GS_{\bf k}$ say that $i$ is {\it a descent} index if $\sigma(i)>\sigma(i+1)$ and $i<K$ or $i=K,$ where $K=k_1+\cdots+k_n.$  Let
$$A_{{\bf k},i}=|\{\sigma\in \GS_{\bf k}: |\mathrm{des}(\sigma)|=i\}|$$
be the number of Stirling permutations that have $i$ descents (here $\rm des(\sigma)$ is a set of descent indices of $\sigma$). The number $A_{{\bf k},i}$ is called {\it Eulerian number} 
and the polynomial $\sum_{i=1}^{n} A_{{\bf k},i} x^i$ is {\it Eulerian polynomial}.  Since all copies of every element $j$ $(1\le j\le n)$ contain at most one descent index, it is clear that  $A_{{\bf k},i}=0$ if $i>n.$ All copies of the greatest element cannot be separated and can be put in any of $k_1 + \cdots + k_{n-1} + 1$ spaces between the other elements; this provides that
$$
|\GS_{\mathbf{k}}| = \prod_{i = 1}^{n -1} (k_1 + \cdots + k_i + 1).
$$

Define the rational functions $ G_{\mathbf{k}}(x)$ and $\G_{\mathbf{k}}(x) $ as
$$
G_{\mathbf{k}}(x) =  \frac{\sum_{i = 1}^{n} A_{\mathbf{k}, i} x^i }{(1-x)^{K + 1}}\quad 
\mbox{ and }\quad 
\G_{\mathbf{k}}(x) = \frac{\sum_{i = K - n + 1}^{K} A_{\mathbf{k}, K + 1 - i} x^i }{(1-x)^{K + 1}}.$$
These functions can be presented as formal power series of $x$ and  
let  $B_{\mathbf k}(m)$, $\B_{\bf k}(m)$ be their corresponding coefficients:  
$$
G_{\mathbf{k}}(x) = \sum_{m = 0}^{\infty} B_{\mathbf{k}}(m) x^m,  \qquad \G_{\mathbf{k}}(x) = 
\sum_{m = 0}^{\infty} \B_{\mathbf{k}}(m) x^m.
$$
In fact, these coefficients are polynomials in $m.$ Note that the series above yield
$$
B_{\mathbf{k}}(0) = 0 \text{ and } \B_{\mathbf{k}}(0) = \cdots = \B_{\mathbf{k}}(K - n) = 0.
$$
We call the polynomials $B_{\mathbf k}(m)$ and $\B_{\bf k}(m)$ {\it Stirling polynomials}. 
The reason for such terminology is that 
$$B_{\bf k}(m)=\stir{n+m}{m} \text{ and }\B_{\bf k}(m)=\ctir{m}{m-n},$$ 
where $\ctir{i}{j}$ and $\stir{i}{j}$ are Stirling numbers of the first and second kinds, if $k_i=2$ for all $i=1,2,\ldots,n$ (\cite{stirpol}).  

The aim of our paper is to give combinatorial interpretations of Stirling polynomials $B_{\mathbf{k}}(m), \B_{\mathbf{k}}(m)$ for all $k_1, \ldots, k_n$. 
Our approach is the following.
\begin{itemize}
\item Firstly, we apply the theory of $P$-partitions \cite{enum} and construct posets, we call them {\it $\mathbf{k}$-Stirling posets} $P_{\mathbf{k}}$, whose order polynomials $\Omega(P_{\mathbf{k}}, m), \overline\Omega(P_{\mathbf{k}}, m)$ equal to $B_{\mathbf{k}}(m), \B_{\mathbf{k}}(m),$ respectively. 

\item Next, we introduce the {\it $\mathbf{k}$-Stirling numbers of first and second kinds} $\ctirg{n}{m}, \stirg{n}{m}$, for which
$$B_{\mathbf{k}}(m) = \stirg{\ell + m}{m} \text{ and } \B_{\mathbf{k}}(m) = \ctirg{m}{m - \ell},$$ 
where $\ell = \ell(\mathbf{n})$ is a number of components of  $\mathbf{n}$ with multiplicities greater than $1,$ 
$$\ell=|\{ i\ |\ k_i>1, i=1,\ldots,n\}|.$$
Combinatorial meanings of $\stirg{n}{m}$ and $\ctirg{n}{m}$ are related to partitions of sets and permutation records. \\
If $\mathbf{k} = (1, 2, \ldots, 2)$, we call the $\mathbf{k}$-Stirling numbers as the {\it Stirling numbers of odd type}.\footnote{In \cite{knuth-opt} these numbers of the second kind were denoted as half-integer Stirling numbers.}
The case $\mathbf{k} = (1, \ldots, 1, 2, \ldots,1, \ldots, 1, 2)$ yields that the $\mathbf{k}$-Stirling numbers naturally generalize the central factorial numbers. 
\end{itemize}

\

{\bf Related work.} Gessel and Stanley \cite{stirpol} were first who introduced the notion of Stirling permutations and presented combinatorial interpretations for the coefficients of the polynomial $(1-x)^{2n+1}\sum \stir{n+m}{m} x^m$.  

Brenti \cite{brenti}, \cite{brenti-h} studied Stirling permutations in general case for all $k_i$. He has obtained algebraic properties of Stirling polynomials and proved that $B_{\mathbf{k}}(m + 1)$ is a Hilbert polynomial. Note that $\Omega(P, m + 1)$ is a Hilbert polynomial for any poset $P$, and therefore our construction of the ${\bf k}$-Stirling poset implies the same property for $B_{\mathbf{k}}(m + 1)$. 

For $k_1 = \cdots = k_n$, the $\mathbf{k}$-Stirling poset was introduced by Klingsberg and Schmalzried \cite{klingsberg}. Park \cite{park-p} also studied this case with extensions to $q$-Stirling numbers. 

Similar problems have been studied for the Legendre-Stirling and the Jacobi-Stirling numbers and polynomials. Egge \cite{egge} has presented a theory concerning the Legendre-Stirling permutations. Gessel, Lin and Zeng \cite{ges-jac} have applied the theory of $P$-partitions for the Jacobi-Stirling polynomials. In our notation, their combinatorial structures apparently work with $k_i = 1,2$ (for all $1 \le i \le n$), where no two consecutive $k_i$ equal to $1$.

\

{\bf Examples. }
\begin{center}
\begin{tabular}{|c|c|c|}
\hline
$\mathbf{k}$ & $B_{\mathbf{k}}(m)$ & $\B_{\mathbf{k}}(m)$\\
\hline
\hline
$(k_1, \ldots, k_n) = (1, \ldots, 1)$ & $m^n$ & $m^n$\\
$(k_1) = (k)$ & $\binom{k + m  -1}{k}$ & $\binom{m}{k}$\\
$(k_1, \ldots, k_n) = (1, \ldots, 1, 2)$ & $\sum_{i = 1}^{m} i^n$ & $\sum_{i = 1}^{m - 1} i^n$\\
$(k_1, \ldots, k_n) = (2, \ldots, 2)$ & $\stir{n+m}{m}$ & $\ctir{m}{m-n}$\\
$(k_1, \ldots, k_{2n}) = (1, 2, \ldots, 1, 2)$ & $T(2n+2m, 2m)$ & $t(2m, 2m-2n)$\\
\hline
\end{tabular}
\end{center}

Here $T(2i, 2j), t(2i, 2j)$ are the central factorial numbers \cite{riordan}.

\section{General properties of Stirling polynomials}

\begin{thm}\label{recmain}
Let $m$ be a positive integer. Then 
\begin{itemize}
\item 
$B_{\mathbf{k}}(m), \B_{\mathbf{k}}(m)$ are both polynomials in $m$ of degree $K$ with leading coefficients $|\GS_{\mathbf{k}}|/K!$ and 
$$B_{\mathbf{k}}(0) = B_{\mathbf{k}}(-1) = \cdots = B_{\mathbf{k}}(-K + n) = 0,\quad
B_{\mathbf{k}}(m) = (-1)^{K}\B_{\mathbf{k}}(-m).$$
\item if $k_n > 1,$  then 
\begin{align}
B_{\mathbf{k}}(m) = \sum_{i = 0}^m i B_{\mathbf{k} \setminus k_n}(i) \binom{k_n+m-i-2}{k_n - 2},\quad
\B_{\mathbf{k}}(m) = \sum_{i = 0}^{m - 1} i \B_{\mathbf{k} \setminus k_n}(i) \binom{m - i - 1}{k_n - 2},\label{5'}
\end{align}
where $\mathbf{k} \setminus k_n = (k_1, \ldots, k_{n-1})$.

\item 
$B_{\emptyset}(m)=1, B_{\mathbf{k}}(0) = 0$; and
\begin{equation}
\label{main}
B_{\mathbf{k}}(m) = \begin{cases}
B_{\mathbf{k}}(m - 1) + B_{\mathbf{k'}}(m), &\text{if $k_{n} > 1$;}\\
m B_{\mathbf{k'}}(m), &\text{if $k_{n} = 1$,}
\end{cases}
\end{equation}

\item 
$\B_{\emptyset}(m)=1, \B_{\mathbf{k}}(0) = 0$; and
\begin{equation}
\label{main'}
\B_{\mathbf{k}}(m) = \begin{cases}
\B_{\mathbf{k}}(m - 1) + \B_{\mathbf{k'}}(m - 1), &\text{if $k_{n} > 1$;}\\
m \B_{\mathbf{k'}}(m), &\text{if $k_{n} = 1$,}
\end{cases}
\end{equation}
where $\mathbf{k'} = (k_1, \ldots, k_{n-1}, k_n - 1)$.
\end{itemize}
\end{thm}

To prove Theorem \ref{recmain} we need the following
\begin{lm}
\label{thm1}
Let $\mathbf{k} \setminus k_n = (k_1, \ldots, k_{n-1})$.
The recurrence for $A_{\mathbf{k}, i}$ is given by: 
\begin{eqnarray} 
A_{\mathbf{k},i} = i \cdot A_{\mathbf{k} \setminus k_n, i} + (k_1 + \ldots + k_{n - 1} + 1 - (i-1)) \cdot A_{\mathbf{k} \setminus k_n,i-1},\label{1n}
\end{eqnarray}
 with $A_{(k), 1} = 1$ and $A_{\mathbf{k}, i} = 0$ if $i = 0$ or $i > n.$

The following differential equations hold for $G_{\mathbf{k}}(x), \G_{\mathbf{k}}(x)$:
\begin{equation}
G_{\mathbf{k}}(x) = \dfrac{x}{(1-x)^{k_n-1}} \frac{d(G_{\mathbf{k} \setminus k_n}(x))}{dx},\label{4n}
\end{equation}
\begin{equation}
\G_{\mathbf{k}}(x) = \dfrac{x^{k_n}}{(1-x)^{k_n-1}} \frac{d(\G_{\mathbf{k} \setminus k_n}(x))}{dx}.\label{4n1}
\end{equation}
\end{lm}

\begin{proof} The proof of \eqref{1n} is standard.  
Stirling permutations of the multiset $\{1^{k_1},\ldots,n^{k_n}\}$ can be obtained from Stirling permutations of the multiset $\{1^{k_1},\ldots,(n-1)^{k_{n-1}}\}$ by inserting the block $n^{k_{n}}$ in any of $k_1+\cdots+k_{n-1}+1$ spaces between the elements. 
Let $\sigma$ be the permutation of $\{1^{k_1},\ldots,(n-1)^{k_{n-1}}\}$
and $\sigma^{(t)}$ be the corresponding permutation of $\{1^{k_1},\ldots,n^{k_n}\},$ where the block $n^{k_{n}}$ is inserted to the $t$-th place of $\sigma$. If $t$ is a descent index of $\sigma,$ then $\des\sigma^{(t)}=\des\sigma.$ 
If $t$ is not a descent index of $\sigma$, then $\des\sigma^{(t)}=\des\sigma+1.$ In other words,  the block $n^{k_{n}}$ can be inserted in any of $i$ descents of $A_{\mathbf{k} \setminus k_n,i}$ permutations without producing a new descent; or it creates a new descent at any of $(k_1 + \ldots + k_{n - 1} + 1 - (i-1))$ positions (with no descent) of $A_{\mathbf{k} \setminus k_n,i-1}$ permutations. So,  \eqref{1n}  is proved. 

By (\ref{1n}) we have 
\begin{align*}
G_{\mathbf{k}}(x) &= \dfrac{\sum_{i = 1}^n A_{\mathbf{k},i} x^i}{(1-x)^{K+1}} \\
&= \frac{x}{(1 - x)^{k_n - 1}} \frac{\sum_{i = 1}^n (i A_{\mathbf{k} \setminus k_n, i} x^{i-1} + (k_1 + \cdots + k_{n-1} + 2 - i) A_{\mathbf{k} \setminus k_n, i-1} x^{i-1} )}{(1-x)^{k_1 + \cdots + k_{n-1} + 2}}\\
&= \frac{x}{(1 - x)^{k_n - 1}} d\left(\frac{\sum_{i = 1}^{n-1} A_{\mathbf{k} \setminus k_n, i} x^i}{(1-x)^{k_1 + \cdots + k_{n-1} + 1}} \right)/dx\\
&= \dfrac{x}{(1-x)^{k_n-1}} \frac{d(G_{\mathbf{k} \setminus k_n}(x))}{dx}.
\end{align*}
Note that  
$$
\G_{\mathbf{k}}(x) = (-1)^{K + 1} G_{\mathbf{k}}(1/x).
$$
Thus, from equation \eqref{4n}
\begin{equation*}
G_{\mathbf{k}}(1/x) = (-1)^{k_n-1}\dfrac{x^{k_n-2}}{(1-x)^{k_n-1}} \frac{d(G_{\mathbf{k} \setminus k_n}(1/x))}{d(1/x)}.
\end{equation*}
Therefore, 
\begin{equation*}
\G_{\mathbf{k}}(x) = -\dfrac{x^{k_n-2}}{(1-x)^{k_n-1}} \frac{d(\G_{\mathbf{k} \setminus k_n}(x))}{d(1/x)} = \dfrac{x^{k_n}}{(1-x)^{k_n-1}} \frac{d(\G_{\mathbf{k} \setminus k_n}(x))}{dx}.
\end{equation*}
\end{proof}

\begin{proof}[Proof of Theorem \ref{recmain}]
By \eqref{4n}, \eqref{4n1} we have
\begin{align*}
\sum_{m = 0}^{\infty} B_{\mathbf{k}}(m) x^m &= \frac{1}{(1-x)^{k_n - 1}} \sum_{j = 0}^{\infty} j B_{\mathbf{k} \setminus k_n}(j) x^j,\\
\sum_{m = 0}^{\infty} \B_{\mathbf{k}}(m) x^m &= \frac{1}{(1-x)^{k_n - 1}} \sum_{j = 0}^{\infty} j \B_{\mathbf{k} \setminus k_n}(j) x^{k_n + j - 1}.
\end{align*}
To obtain  \eqref{5'}, it remains to use the well known relation 
$\frac{1}{(1-x)^{k_n - 1}} = \sum_{i  = 0}^{\infty} \binom{k_n + i - 2}{k_n - 2} x^i$.

If $k_n > 1,$ then equations \eqref{4n}, \eqref{4n1} can be written as 
\begin{align*}
G_{\mathbf{k}}(x) &= \dfrac{x}{(1-x)^{k_n-1}} \frac{d(G_{\mathbf{k} \setminus k_n}(x))}{dx}\\
&= \frac{1}{(1-x)}\dfrac{x}{(1-x)^{k_n-2}} \frac{d(G_{\mathbf{k} \setminus k_n}(x))}{dx} \\
&= \frac{1}{(1-x)} G_{\mathbf{k'}}(x),
\end{align*}
\begin{align*}
\G_{\mathbf{k}}(x) &= \dfrac{x^{k_n}}{(1-x)^{k_n-1}} \frac{d(\G_{\mathbf{k} \setminus k_n}(x))}{dx}\\
&= \frac{x}{(1-x)}\dfrac{x^{k_n - 1}}{(1-x)^{k_n-2}} \frac{d(\G_{\mathbf{k} \setminus k_n}(x))}{dx} \\
&= \frac{x}{(1-x)} \G_{\mathbf{k'}}(x).
\end{align*}
Thus, we have
$$
G_{\mathbf{k}}(x) = \frac{1}{(1-x)} G_{\mathbf{k'}}(x),
\quad
\G_{\mathbf{k}}(x) = \frac{x}{(1-x)} \G_{\mathbf{k'}}(x),
$$
which provide us the first cases of recurrences \eqref{main}, \eqref{main'}.

If $k_n = 1$, then the second cases of \eqref{main}, \eqref{main'}  are easy consequences of equations \eqref{4n}, \eqref{4n1}.

Let us now prove by induction on $K$ that for every multiset $\mathbf{n} = \{1^{k_1}, \ldots, n^{k_n} \}$ having $K$ elements,
the polynomial $B_{\mathbf{k}}(m)$ is a polynomial in $m$ of degree $K$ with the leading coefficient  $|\GS_{\mathbf{k}}|/K!$ and $$B_{\mathbf{k}}(0) = B_{\mathbf{k}}(-1) = \cdots = B_{\mathbf{k}}(-K + n) = 0.$$

If $K = 0$, then $B_{\emptyset}(m) = 1$; if $K = 1,$ then $B_{\mathbf{k}}(m) = m$.

Suppose that the statement is true for all multisets having less than $K$ elements and let $\mathbf{n} = \{1^{k_1}, \ldots, n^{k_n} \}$ be any multiset having $K$ elements. 

If $k_n > 1$, then $B_{\mathbf{k'}}(m)$ is a polynomial in $m$ of degree $K - 1$ with the leading coefficient $$a = |\GS_{\mathbf{k'}}|/(K - 1)! = |\GS_{\mathbf{k}}|/K!$$ and 
$$B_{\mathbf{k'}}(0) = B_{\mathbf{k'}}(-1) = \cdots = B_{\mathbf{k'}}(-K + 1 + n) = 0.$$ Hence, by the recurrence \eqref{main} if $m$ is any positive integer, then 
$$B_{\mathbf{k}}(m) -B_{\mathbf{k}}(m - 1) = B_{\mathbf{k'}}(m)
\text{ or } 
B_{\mathbf{k}}(m) = \sum_{i = 1}^{m} B_{\mathbf{k'}}(i).$$ 
Therefore, $B_{\mathbf{k}}(m)$ is a polynomial in $m$ of degree $K$ with the leading coefficient $$a/K = |\GS_{\mathbf{k}}|/K!.$$
Hence, $$B_{\mathbf{k}}(m) -B_{\mathbf{k}}(m - 1) = B_{\mathbf{k'}}(m)$$ for any $m.$  So, 
$$B_{\mathbf{k}}(0) - B_{\mathbf{k}}(-m - 1)  = \sum_{i = -m}^{0} B_{\mathbf{k'}}(i).$$
By the definition, $B_{\mathbf{k}}(0) = 0$ and hence $$B_{\mathbf{k}}(0) = B_{\mathbf{k}}(-1) = \cdots = B_{\mathbf{k}}(-K + n) = 0.$$ 

If $k_n = 1,$ then $B_{\mathbf{k'}}(m)$ is a polynomial in $m$ of degree $K - 1$ with the leading coefficient $$a = |\GS_{\mathbf{k'}}|/(K - 1)! = |\GS_{\mathbf{k}}|/K!$$ and 
$$B_{\mathbf{k'}}(0) = B_{\mathbf{k'}}(-1) = \cdots = B_{\mathbf{k'}}(-K + 1 + n - 1) = 0.$$ By the recurrence relation \eqref{main}, $B_{\mathbf{k}}(m) = m B_{\mathbf{k'}}(m).$ Therefore, $B_{\mathbf{k}}(m)$ is a polynomial in $m$ of degree $K$ with the leading coefficient $$a = |\GS_{\mathbf{k'}}|/(K - 1)! =  |\GS_{\mathbf{k}}|/K!$$ and $$B_{\mathbf{k}}(0) = B_{\mathbf{k}}(-1) = \cdots = B_{\mathbf{k}}(-K + n) = 0.$$

Now, if we take a new polynomial $$f_{\mathbf{k}}(m) = (-1)^K B_{\mathbf{k}}(-m),$$ then $f_{\emptyset}(m) = 1,$ $f_{\mathbf{k}}(0) = 0$ and the recurrence \eqref{main} gives 
$$f_{\mathbf{k}}(m) = f_{\mathbf{k}}(m - 1)  + f_{\mathbf{k'}}(m - 1), \text{ if } k_n > 1$$ and $$f_{\mathbf{k}}(m) = m f_{\mathbf{k'}}(m), \text{ if } k_n = 1,$$ 
which implies that $f_{\mathbf{k}}(m) = \B_{\mathbf{k}}(m).$
\end{proof}

\section{Stirling polynomials as order polynomials}
Suppose that $P$ is a finite labeled partially ordered set  with the partial 
order $<_p$. 

\begin{dfn} Let $\Omega(P,m)$ be the
number of order-preserving maps 
$\sigma: P \to \{1,\dots, m\}$ 
and $\overline\Omega(P,m)$ be the
number of strict order-preserving maps 
$\overline{\sigma}: {P} \to \{1,\dots, m\},$
i.e.,
$\text{if } x <_p y \text{ then } \sigma(x) \le \sigma(y)$ and $\overline{\sigma}(x) <\overline{\sigma}(y)$.
\end{dfn}

It is known that $\Omega({P},m),$ $\overline\Omega({P},m)$ are polynomials in $m$ called \textit{order polynomials} and  $\Omega({P},m) = (-1)^{|{P}|}\overline\Omega({P},-m).$
(see \cite{enum}.)

Let us call an $s$-tuple $(q_1,\ldots q_{s})$ by an {\it $s$-serie with end} $q_{s}$ if $q_1=\ldots= q_{s-1}=1$ and $q_{s}>1$; or just by an {\it $s$-serie} if such $q_{s}$ does not exist. 
We say that the multiset ${\bf n}$ (or $n$-tuple  ${\bf k}$) has 
$$\text{{\it length} $\ell({\bf n})=\ell$ and {\it weight} $w({\bf n})=(a_1,\ldots,a_{\ell}; t_1,\ldots,t_{\ell}; a)$} $$
if ${\bf k}$ can be presented as a sequence of $a_i$-series with ends $t_i$ and $a$-serie:  
$$
(k_1, \ldots, k_n) = (\underbrace{1, \ldots, 1}_{a_1 - 1 \text{ ones}}, t_1, \ldots, \underbrace{1, \ldots, 1}_{a_{\ell} - 1 \text{ ones}}, t_{\ell}, \underbrace{1, \ldots, 1}_{a \text{ ones}}),
$$
where $a_i>0, t_i > 1$ for all $1\le i\le \ell$ and $a \ge 0.$  

For example, if $n=10$ and $k_1=1,k_2=1,k_3=3,k_4=2,k_5=1,k_6=1,k_7=1,k_8=2,k_9=5,k_{10}=6,$ then 
$${\bf k}=(1,1,3,2,1,1,1,2,5,6) \sim   
(1,1,3) \;(2) \; (1,1,1,2)\; (5)\; (6)$$
is a sequence of $a_i$-series, where $a_1=3, a_2=1,a_3=4, a_4=1, a_5=1,$ with ends $t_1=3, t_2=2, t_3=2, t_4=5, t_5=6.$ So, in this example, the multiset ${\bf n}$ has length $5$ and  weight $(3,1,4,1,1; 3,2,2,5,6; 0).$

Suppose that ${\bf n}=\{1^{k_1},\ldots,n^{k_n}\}$ has weight  $(a_1,\ldots,a_{\ell}; t_1,\ldots,t_{\ell}; a)$. Set 
$$
s_0 = 0, s_{i} = s_{i - 1} + a_i + t_i - 1 \text{ or } s_{i} = \sum_{j = 1}^{i} (a_j + t_j - 1) \text{ for } 1 \le i \le \ell.
$$
Define the {\it $\mathbf{k}$-Stirling poset} ${P}_{\mathbf{k}}$ by the following diagram: 

\begin{center}
\includegraphics{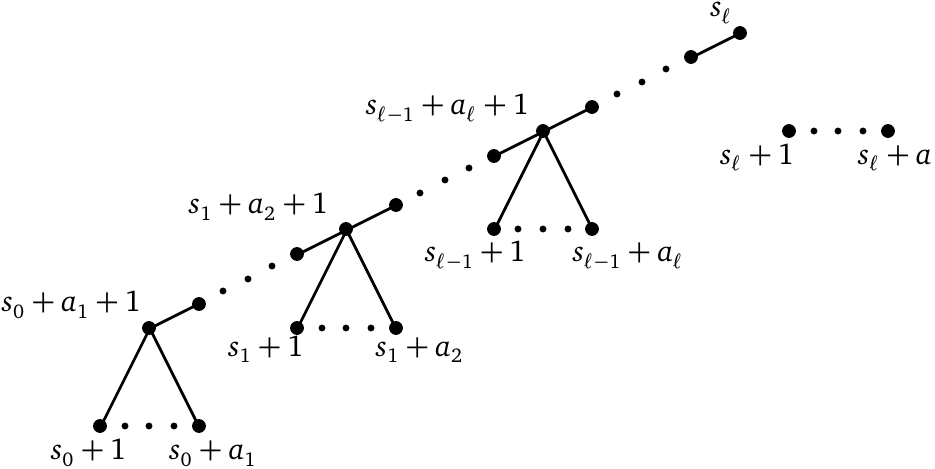}

{\bf Fig. 1.} The $\mathbf{k}$-Stirling poset $P_{\bf k}$.
\end{center}

\noindent Here the elements labeled by $s_{i-1} + 1, \ldots, s_{i-1} + a_i$ collapse to one, if $a_i=1.$  If $a > 0,$ then the elements with labels  $s_{\ell} + 1, \ldots, s_{\ell} + a$ are incomparable with other elements of $P_{\mathbf{k}}$.  
For example, if ${\bf k}=(1,1,2,3,1,2,2,1,1,1),$ then ${P}_{\mathbf{k}}$ is

\begin{center}
\includegraphics{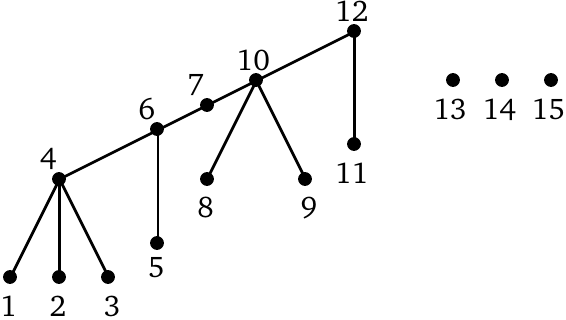}

{\bf Fig. 2.} The poset $P_{(1,1,2,3,1,2,2,1,1,1)}$.
\end{center}

\begin{thm}\label{poset}
 $B_{\mathbf{k}}(m) = \Omega({P}_{\mathbf{k}},m) \text{ and } \B_{\mathbf{k}}(m) = \overline\Omega({P}_{\mathbf{k}},m).$
\end{thm}

\begin{proof}
Let $v$ be the maximal label in ${P}_{\mathbf{k}}$ and $\mathbf{k} \setminus k_n = (k_1, \ldots, k_{n-1})$. 

\textit{Case 1.} If $k_n = 1$, then $v$ is incomparable with the other elements and $\sigma(v)$ can take any of $m$ values. Thus, $\Omega({P}_{\mathbf{k}},m) = m \Omega({P}_{\mathbf{k} \setminus k_n},m).$

\textit{Case 2.} If $k_n > 1$, then $\sigma(v)$ is the maximal value in the map $\sigma$ and two subcases are possible:

(a) if $\sigma(v) \le m-1,$ then the number of maps is equal to $\Omega({P}_{\mathbf{k}},m - 1)$;

(b) if $\sigma(v) = m,$ then the removal of $v$ gives us $\Omega({P}_{\mathbf{k} \setminus k_n},m)$ ways to map the remaining elements.

Hence, in Case 2 we have 
$$
\Omega({P}_{\mathbf{k}},m) = \Omega({P}_{\mathbf{k} \setminus k_n},m) + \Omega({P}_{\mathbf{k}},m - 1).
$$ 
So, $\Omega({P}_{\mathbf{k}},m)$ satisfies the same recurrence relation as (\ref{main}) of $B_{\mathbf{k}}(m)$ and it is easy to check that the initial values are also equal.

According to the reciprocity of order polynomials,
$\B_{\mathbf{k}}(m) = \overline\Omega({P}_{\mathbf{k}},m).$
\end{proof}

Particular cases of our construction were known before. 
For instance, the poset that induces Stirling numbers of the second kind $B_{(2, \ldots, 2)}(m) = \stir{n+m}{m}$, i.e., $\Omega({P}, m) = B_{(2, \ldots, 2)}(m)$, is the following \cite{klingsberg, park-p}:

\noindent\begin{center}
\begin{tabular}{c}
\includegraphics{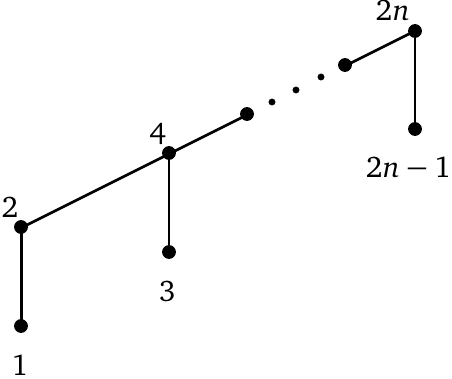}
\end{tabular}

{\bf Fig. 3.} The poset $P_{(2,\ldots,2)}$.
\end{center}

It gives 
$$
\Omega({P_{(2,\ldots,2)}}, m) = \sum_{1 \le \sigma(2) \le \cdots \le \sigma(2n) \le m} \sigma(2) \cdots \sigma(2n),$$
$$
\overline\Omega({P_{(2,\ldots,2)}}, m) = \sum_{1 \le \overline\sigma(2) < \cdots < \overline\sigma(2n) < m} \overline\sigma(2) \cdots \overline\sigma(2n).
$$

The poset ${P}_{(k, \ldots, k)}$ was constructed in \cite{klingsberg}.  

\section{Stirling polynomials as numbers of set partitions and permutation records}
\label{str}

Let $[n] = \{1, \ldots, n \}$.
For a family $\mathcal{F} = \{B_1, \ldots, B_m\}$ of nonempty sets (or multisets) let 
$\min(B_i)$ be the minimal element of $B_i$ and
$\min(\mathcal{F}) = \{\min(B_1), \ldots, \min(B_m)  \}$. We write $a \sim_{\mathcal{F}} b$ if $a, b \in B_j$ for some $j.$ 
We define a multiset $\mathrm{multiset}(\mathcal{F})$ as a merge sum of multisets:
$$\mathrm{multiset}(\mathcal{F}) = \uplus_{i = 1}^m B_i.$$ 
In other words, $a \sim_{\mathcal{F}} b$ if $a,b$ are in the same set 
and $\mathrm{multiset}(\mathcal{F})$ is a multiset of all elements of $B_i$.
For example, if $\mathcal{F} = \{\{1,3, 6 \}, \{ 2, 3, 3, 5\}, \{ 2, 4, 6 \} \},$ then $2 \sim_{\mathcal{F}} 3,$
$2 \sim_{\mathcal{F}} 4,$ but $2\not{\!\!\sim_{\mathcal{F}}} 1$  
and $\mathrm{multiset}({\mathcal{F}}) = \{ 1, 2^2, 3^3, 4, 5, 6^2\}.$

For the given set $U$ consider nonempty sets (or multisets) $B_1, \ldots, B_m$, where $B_i \subseteq U$ $(1 \le i \le m)$. Say that the family $\mathcal{F} = \{B_1, \ldots, B_m\}$ is \textit{segmented} if
for all $a < b < c$ ($a, b, c \in U$) the condition  $a \sim_{\mathcal{F}} c$ implies that $b \in \min(\mathcal{F})$.
For example, if $U = \{1, \ldots, 8\}$, then $\{\{ 1, 2\}, \{3, 3, 6 \}, \{ 4\}, \{5 \}, \{7, 7, 8\}\}$ is segmented and $\{\{ 1, 2\}, \{3, 3, 6 \}, \{ 4, 5 \}, \{7, 7, 8\}\}$ is not segmented, because $3 \sim_{\mathcal{F}} 6$ and $3 < 5 < 6$, $5 \not\in \min(\{\{ 1, 2\}, \{3, 3, 6 \}, \{ 4, 5 \} ,\{7, 7, 8\}\}) =\{1, 3, 4, 7 \}$. 

Let $\text{lmin}(\sigma)$ be the set of left-to-right minima of a permutation $\sigma$.  

Suppose now that $n,m$ ($n \ge m$) are given positive integers, $U = [n+1]$ and let $\mathbf{k}$ be any tuple with weight $(a_1,\ldots,a_{n-m}; t_1,\ldots,t_{n-m}; 0),$ which represents the type of some multiset of length $n - m$. Let $M = \max(a_1, \ldots, a_{n - m})$.

\begin{dfn}\label{smain}
A {\it $\mathbf{k}$-partition system} of $[n]$ into $m$ blocks is an ordered $(M+1)$-tuple $(\pi_0, \pi_1, \ldots, \pi_M)$ which satisfies the following properties:

(i) $\pi_1, \ldots, \pi_M$ are partitions of $[n]$ into nonempty blocks;

(ii) $\min(\pi_1) = \cdots = \min(\pi_M)$ and $|\mathrm{min}(\pi_1)| = m$;

(iii) if $\{x_1, \ldots, x_{n - m} \} = [n] \setminus \min(\pi_1)$ so that $x_1 < \cdots < x_{n - m}$, then for all $i$ $(1 \le i \le n - m)$ and $j > a_i$, we have $x_i \sim_{\pi_j} 1$;

(iv) $\pi_0$ is a segmented family of nonempty multisets such that $\min(\pi_0) = \min(\pi_1)$, $x_1 \sim_{\pi_0} 1$ and $\mathrm{multiset}(\pi_0) \subseteq \min(\pi_1) \uplus \{x_1, x_2^{t_1 -2}, \ldots, x_{n-m}^{t_{n-m-1} - 2}, (n+1)^{t_{n - m} - 2} \}$.
\end{dfn}

\begin{dfn}\label{smain1}
A {\it $\mathbf{k}$-permutation system} of $[n]$ having $m$ left-to-right minima is an ordered $(M + 1)$-tuple $(\sigma_0, \sigma_1, \ldots, \sigma_M)$ which satisfies the following properties:

(i) $\sigma_1, \ldots, \sigma_M$ are permutations of $[n]$;

(ii) $\text{lmin}(\sigma_1) = \cdots = \text{lmin}(\sigma_M)$ and $|\mathrm{lmin}(\sigma_1)| = m$;

(iii) if $\{x_1, \ldots, x_{n - m} \} = [n] \setminus \text{lmin}(\sigma_1)$ so that $x_1 < \cdots < x_{n - m}$, then for any $i$ $(1 \le i \le n - m)$ and $j > a_i$, if $\sigma_j(p) = x_i$, then $\sigma_j(q) > x_i$ for all $q > p$;

(iv) $\sigma_0$ is a segmented family of nonempty sets such that $\min(\sigma_0) = \mathrm{lmin}(\sigma_1)$, $x_1 \sim_{\sigma_0} 1$ and $\mathrm{multiset}(\sigma_0) = \mathrm{lmin}(\sigma_1) \uplus \{x_1, x_2^{t_1 -2}, \ldots, x_{n-m}^{t_{n-m-1} - 2}, (n+1)^{t_{n - m} - 2} \}$.
\end{dfn}

\begin{dfn}
Let $\stirg{n}{m}$ be the number of $\mathbf{k}$-partition systems of $[n]$ into $m$ blocks and $\ctirg{n}{m}$ be the number of  $\mathbf{k}$-permutation systems of $[n]$ having $m$ left-to-right minima. 
\end{dfn}

{\bf Example} (see Definition 2.) Let $n = 10, m = 5,$ $\mathbf{k} = (1, 1, 3, 2, 1, 1, 1, 2, 5, 6)$. Then 
$(a_1, a_2, a_3, a_4, a_5) = (3,1,4,1,1)$ and $(t_1, t_2, t_3, t_4, t_5) = (3, 2, 2, 5, 6)$ and $(\pi_0, \pi_1, \pi_2, \pi_3, \pi_4)$ is a $\mathbf{k}$-partition system of $\{1, \ldots, 10\}$ into $5$ blocks, where

\begin{tabular}{clllll}
$\pi_1=$ & $\{\mathbf{1}, 7 \}$ & $\{\mathbf{2}, 3 \}$ & $\{\mathbf{4}, 5, 10 \}$ & $\{\mathbf{6} \}$ & $\{\mathbf{8}, 9 \}$;\\ 
$\pi_2=$ & $\{\mathbf{1}, 3, 5, 9, 10 \}$ & $\{\mathbf{2}, 7 \}$ & $\{\mathbf{4} \}$ & $\{\mathbf{6} \}$ & $\{\mathbf{8} \}$;\\ 
$\pi_3=$ & $\{\mathbf{1}, 5, 9, 10 \}$ & $\{\mathbf{2}, 3 \}$ & $\{\mathbf{4},7 \}$ & $\{\mathbf{6} \}$ & $\{\mathbf{8} \}$;\\ 
$\pi_4=$ & $\{\mathbf{1},3, 5, 9, 10 \}$ & $\{\mathbf{2} \}$ & $\{\mathbf{4} \}$ & $\{\mathbf{6},7 \}$ & $\{\mathbf{8} \}$;\\ 
$\pi_0=$ & $\{\mathbf{1}, 3\}$ & $\{\mathbf{2} \}$ & $\{\mathbf{4}, 5 \}$ & $\{\mathbf{6} \}$ & $\{\mathbf{8} \}$.\\ 
\end{tabular}

Another possible configuration is

\begin{tabular}{clllll}
$\pi_1=$ & $\{\mathbf{1}, 7 \}$ & $\{\mathbf{2}, 3 \}$ & $\{\mathbf{4}, 5, 9 \}$ & $\{\mathbf{6}, 8 \}$ & $\{\mathbf{10}\}$;\\ 
$\pi_2=$ & $\{\mathbf{1}, 3, 5, 8, 9 \}$ & $\{\mathbf{2}, 7 \}$ & $\{\mathbf{4} \}$ & $\{\mathbf{6} \}$ & $\{\mathbf{10} \}$;\\ 
$\pi_3=$ & $\{\mathbf{1}, 5, 8, 9 \}$ & $\{\mathbf{2}, 3 \}$ & $\{\mathbf{4},7 \}$ & $\{\mathbf{6} \}$ & $\{\mathbf{10} \}$;\\ 
$\pi_4=$ & $\{\mathbf{1},3, 5, 8, 9 \}$ & $\{\mathbf{2} \}$ & $\{\mathbf{4} \}$ & $\{\mathbf{6},7 \}$ & $\{\mathbf{10} \}$;\\ 
$\pi_0=$ & $\{\mathbf{1}, 3\}$ & $\{\mathbf{2} \}$ & $\{\mathbf{4}, 5 \}$ & $\{\mathbf{6} \}$ & $\{\mathbf{10}, 11^3 \}$.\\ 
\end{tabular}

\

Suppose now that $n = 4, m = 2,$ $\mathbf{k} = (1, 3, 1, 4)$. Then $(a_1, a_2)= (2, 2),$ $(t_1, t_2) = (3, 4)$ and we list the ways to form all $27$ $\mathbf{k}$-partition systems $(\pi_0, \pi_1, \pi_2)$ of $\{1, 2, 3, 4\}$ into $2$ blocks.

Ways 1-16: \ \
\begin{tabular}{cll}
$\pi_1=$ & $\{\mathbf{1}, {\textcolor{gray}{3, 4}}\}$ & $\{\mathbf{2}, {\textcolor{gray}{3, 4}}\}$;\\ 
$\pi_2=$ & $\{\mathbf{1}, {\textcolor{gray}{3, 4}}\}$ & $\{\mathbf{2}, {\textcolor{gray}{3, 4}} \}$;\\ 
$\pi_0=$ & $\{\mathbf{1}, 3\}$ & $\{\mathbf{2} \}$.\\ 
\end{tabular}

put $\textcolor{gray}{3, 4}$: ($4$ ways in $\pi_1$) and ($4$ ways in $\pi_2$);

Ways 17-24:
\begin{tabular}{cll}
$\pi_1=$ & $\{\mathbf{1}, 2, {\textcolor{gray}{4}}\}$ & $\{\mathbf{3}, {\textcolor{gray}{4}}\}$;\\ 
$\pi_2=$ & $\{\mathbf{1}, 2, {\textcolor{gray}{4}}\}$ & $\{\mathbf{3}, {\textcolor{gray}{4}} \}$;\\ 
$\pi_0=$ & $\{\mathbf{1}, 2\}$ & $\{\mathbf{3}, {\textcolor{gray}{4}} \}$.\\ 
\end{tabular}

put $\textcolor{gray}{4}$: ($2$ ways in $\pi_1$) and ($2$ ways in $\pi_2$) and ($2$ ways to put or not $4$ in $\pi_0$);

Ways 25-27:
\begin{tabular}{cll}
$\pi_1=$ & $\{\mathbf{1}, 2, 3\}$ & $\{\mathbf{4}\}$;\\ 
$\pi_2=$ & $\{\mathbf{1}, 2, 3\}$ & $\{\mathbf{4}\}$;\\ 
$\pi_0=$ & $\{\mathbf{1}, 2\}$ & $\{\mathbf{4}, {\textcolor{gray}{5, 5}} \}$.\\ 
\end{tabular}

put $\textcolor{gray}{5, 5}$: $3$ ways $\{\mathbf{4}\}$, $\{\mathbf{4}, 5\}$, $\{\mathbf{4}, 5, 5\}$ (note that the element $3$ can not be put in $\pi_0$ as $\{\mathbf{1}, 2, 3\}$ because the segmented partition has $1 < 2 < 3$, and hence $2$ should be a block minimum).

Therefore, 
$$S_{(1, 3, 1, 4)}(4, 2) = 27.$$

{\bf Example} (see Definition 3.) Let $\mathbf{k} = (1, 1, 3, 2, 1, 1, 1, 2, 3, 3).$ Then
$(a_1, a_2, a_3, a_4, a_5) = (3,1,4,1,1)$ and $(t_1, t_2, t_3, t_4, t_5) = (3, 2, 2, 3, 3)$ and $(\sigma_0, \sigma_1, \sigma_2, \sigma_3, \sigma_4)$ is a $\mathbf{k}$-permutation system of $\{1, \ldots, 10\}$ having $5$ left-to-right minima, where

\begin{tabular}{cl}
$\sigma_1 = $ & $({\mathbf{10}}, {\mathbf{8}}, 9, {\mathbf{7}}, {\mathbf{3}}, 5, 6, 4, {\mathbf{1}}, 2)$;\\
$\sigma_2 = $ & $({\mathbf{10}}, {\mathbf{8}}, {\mathbf{7}}, {\mathbf{3}}, {\mathbf{1}}, 5, 2, 4, 6, 9)$;\\
$\sigma_3 = $ & $({\mathbf{10}}, {\mathbf{8}}, {\mathbf{7}}, {\mathbf{3}}, {\mathbf{1}}, 2, 5, 4, 6, 9)$;\\
$\sigma_4 = $ & $({\mathbf{10}}, {\mathbf{8}}, {\mathbf{7}}, {\mathbf{3}}, {\mathbf{1}}, 2, 4, 5, 6, 9)$;\\
$\sigma_0 = $ & $\{{\mathbf{10}},11 \}, \{ {\mathbf{8}} \}, \{{\mathbf{7}}, 9 \}, \{{\mathbf{3}}, 4 \}, \{{\mathbf{1}}, 2\}$.\\
\end{tabular}

Another possible configuration is

\begin{tabular}{cl}
$\sigma_1 = $ & $({\mathbf{10}}, {\mathbf{8}}, {\mathbf{7}}, 9, {\mathbf{3}}, {\mathbf{1}}, 6, 5, 4, 2)$;\\
$\sigma_2 = $ & $({\mathbf{10}}, {\mathbf{8}}, {\mathbf{7}}, {\mathbf{3}}, {\mathbf{1}}, 2, 4, 5, 6, 9)$;\\
$\sigma_3 = $ & $({\mathbf{10}}, {\mathbf{8}}, {\mathbf{7}}, {\mathbf{3}}, 5, {\mathbf{1}}, 2, 4, 6, 9)$;\\
$\sigma_4 = $ & $({\mathbf{10}}, {\mathbf{8}}, {\mathbf{7}}, {\mathbf{3}}, {\mathbf{1}}, 2, 5, 4, 6, 9)$;\\
$\sigma_0 = $ & $\{{\mathbf{10}},11 \}, \{ {\mathbf{8}}, 9 \}, \{{\mathbf{7}} \}, \{{\mathbf{3}}, 4 \}, \{{\mathbf{1}}, 2\}$.\\
\end{tabular}

\

Suppose that $n = 6, m = 4$, $\mathbf{k} = (1, 3, 1, 4)$. Then $(a_1, a_2) = (2, 2),$ $(t_1, t_2) = (3, 4)$ and we list the ways to form all 9 $\mathbf{k}$-permutations $(\sigma_0, \sigma_1, \sigma_2)$ of $\{1, 2,3,4,5, 6\}$ having $4$ left-to-right minima. 
One can check that from all possible configurations of left-to-right minima the only valid here is $\{6,5,3,1\}$ and we have

\begin{tabular}{ll}
$\sigma_1 = $ & $({\mathbf{6}}, {\mathbf{5}}, {\mathbf{3}}, {\textcolor{gray}{4}}, {\mathbf{1}}, {\textcolor{gray}{4}}, 2, {\textcolor{gray}{4}})$;\\
$\sigma_2 = $ & $({\mathbf{6}}, {\mathbf{5}}, {\mathbf{3}}, {\textcolor{gray}{4}}, {\mathbf{1}}, {\textcolor{gray}{4}}, 2, {\textcolor{gray}{4}})$;\\
$\sigma_0 = $ & $\{{\mathbf{6}}, 7\}, \{{\mathbf{5}}, 7 \}, \{{\mathbf{3}}, 4 \}, \{{\mathbf{1}, 2}\}$;\\
\end{tabular}

there are $3$ ways to place $4$ in $\sigma_1$, $3$ ways to place $4$ in $\sigma_2$, only $1$ way to form $\sigma_0$, which totally gives $9$ possible configurations. So,
$$s_{(1,3,1,4)}(6,4) = 9.$$

\begin{thm}\label{spart} If the multiset ${\bf n}=\{1^{k_1},\ldots, n^{k_n}\}$ has length $\ell$ and weight $(a_1,\ldots,a_{\ell}; t_1,\ldots,t_{\ell}; 0),$ then  
$$B_{\bf k}(m) = \stirg{\ell + m}{m},$$ 
$$\B_{\bf k }(m) = \ctirg{m}{m - \ell}.$$
Moreover, 
\begin{equation}\label{Sk}
 \stirg{\ell + m}{m}= \sum_{1 \le i_1 \le \cdots \le i_{\ell} \le m}
i_1^{a_1} \dots i_{\ell}^{a_{\ell}} 
\binom{t_1 + i_2 - i_1 - 2}{i_2 - i_1} \cdots 
\binom{t_{\ell} + m - i_{\ell} - 2}{m - i_{\ell}},
\end{equation}
\noindent and
\begin{equation}\label{sk}
\ctirg{m}{m - \ell} 
= \sum_{1 \le i_1 < \cdots < i_{\ell} < m}
i_1^{a_1} \dots i_{\ell}^{a_{\ell}} 
\binom{i_2 - i_1 - 1}{t_1 - 2} \cdots 
\binom{m - i_{\ell} - 1}{t_{\ell} - 2}.
\end{equation}
\end{thm}

\begin{proof} Let us first prove that $B_{\mathbf{k}}(m) = \stirg{\ell + m}{m}.$ Set $M = \max(a_1, \ldots, a_{\ell})$. We can fix $m$ minimal elements of $[\ell + m]$ that are common for $(\pi_0, \ldots, \pi_M)$ and consider the elements $\{x_1, \ldots, x_{\ell} \} = [\ell + m] \setminus \min(\pi_1)$, so that $x_1 < \ldots < x_{\ell}$. For any $j$ $(1 \le j \le \ell)$ denote 
$$
i_j = |\{x\ |\ x < x_j, x \in \min(\pi_1) \}|.
$$
Then, according to the property (iii) from Definition \ref{smain} above, the element $x_j$ can be placed in $i_j$ blocks of any of partitions $(\pi_1, \ldots, \pi_{a_j})$, for every $1 \le j \le \ell$. This provides $i_j^{a_j}$ ways to place $x_j$ and totally $i_1^{a_1} \dots i_{\ell}^{a_{\ell}}$ ways to place all the elements $x_1, \ldots, x_{\ell}$, if the minimal elements are fixed.

Consider now the number of ways to form $\pi_0$. The element $x_1$ is already placed with the minimal element $1$. 
Let $x_{\ell + 1} = m + 1$ and $p_1 = 1.$ Then, according to the segmented property, for every $j$ $(2 \le j \le \ell + 1)$, the element $x_j$ can be placed only in multisets whose minimal elements $p_j$ are greater than $x_{j - 1}$; or $x_j$ is placed nowhere. So, for any $j$ $(2 \le j \le \ell + 1)$ there are $i_{j} - i_{j - 1} + 1$ ways to put $t_{j - 1} - 2$ copies of the element $x_j$, which gives $\binom{t_{j-1} - 2 + i_{j} - i_{j - 1}}{i_{j} - i_{j-1}}$ ways.

Thus, for this fixed arrangement of minimal elements, we have 
$$
i_1^{a_1} \dots i_{\ell}^{a_{\ell}} 
\binom{t_1 + i_2 - i_1 - 2}{i_2 - i_1} \cdots 
\binom{t_{\ell} + m - i_{\ell} - 2}{m - i_{\ell}}
$$
ways to form $(M+1)$-tuple $(\pi_0, \ldots, \pi_M)$. Note that it holds for the arbitrary sequence satisfying $1 \le i_1 \le \dots \le i_{\ell} \le m$. So, we have established  (\ref{Sk}).

Iterative use of equation (\ref{5'}) gives the same formula for $B_{\mathbf{k}}(m)$:
\begin{equation*}
B_{\mathbf{k}}(m) = 
\sum_{1 \le i_1 \le \dots \le i_{\ell} \le m}
i_1^{a_1} \dots i_{\ell}^{a_{\ell}} 
\binom{t_1 + i_2 - i_1 - 2}{i_2 - i_1} \dots 
\binom{t_{\ell} + m - i_{\ell} - 2}{m - i_{\ell}}.
\end{equation*}

Let us now prove that $\B_{\mathbf{k}}(m) = \ctirg{m}{m - \ell}.$
To do that we form permutations $(\sigma_1, \ldots, \sigma_M)$, which all have the same set of left-to-right minima. In all these permutations we write the minima in decreasing order  and look at the number of ways to place the remaining elements $x_1, \ldots, x_{\ell}$, satisfying $x_1 < \cdots < x_{\ell}$. Then, according to the property (iii) from Definition \ref{smain1}, the element $x_j$ can be placed to the right of any of $x_j - 1$ elements in permutations $(\sigma_1, \ldots, \sigma_{a_j})$, for any $1 \le j \le \ell$. Note that in the other permutations $(\sigma_{a_j + 1}, \ldots, \sigma_{M})$ the element $x_j$ can be put only in one place -- the rightmost position where the next elements are greater than $x_j$, which satisfies (iii) from Definition \ref{smain1}. This gives $(x_j-1)^{a_j}$ ways to place $x_j$ and totally $(x_1 - 1)^{a_1} \dots (x_{\ell} - 1)^{a_{\ell}}$ ways to place all the elements $x_1, \ldots, x_{\ell}$, if the left-to-right minima are fixed.

Now we count the number of ways to form $\sigma_0$. Let $x_{\ell + 1} = m + 1$ and $p_1 = 1.$ Then, according to the segmented property, for any $j$ $(2 \le j \le \ell + 1)$, the element $x_j$ can be placed only in sets whose minimal elements $p_j$ are greater than $x_{j - 1}$. So, for any $j$ $(2 \le j \le \ell + 1)$ we should put the element $x_j$ in any $t_{j - 1} - 2$ of $x_{j} -1- x_{j - 1}$ vacant sets, which gives $\binom{x_{j} - x_{j - 1} - 1}{t_{j - 1} - 2}$ ways.

Thus, for any fixed arrangement of minimal elements, we have 
$$
(x_1 - 1)^{a_1} \dots (x_{\ell}-1)^{a_{\ell}} 
\binom{x_2 - x_1 - 1}{t_1 - 2} \cdots 
\binom{m + 1 - x_{\ell} - 1}{t_{\ell} - 2}
$$
ways to form $(M+1)$-tuple $(\sigma_0, \ldots, \sigma_M)$. Note that it holds for the arbitrary sequence satisfying $2 \le x_1 < \dots < x_{\ell} \le m$. Therefore, we establish (\ref{sk}):
\begin{align*}
\ctirg{m}{m - \ell} &= \sum_{2 \le x_1 < \cdots < x_{\ell} \le m}(x_1 - 1)_1^{a_1} \dots (x_{\ell}-1)^{a_{\ell}} 
\binom{x_2 - x_1 - 1}{t_1 - 2} \cdots 
\binom{m + 1 - x_{\ell} - 1}{t_{\ell} - 2}.\\&= \sum_{1 \le i_1 < \cdots < i_{\ell} < m}i_1^{a_1} \dots i_{\ell}^{a_{\ell}} 
\binom{i_2 - i_1 - 1}{t_1 - 2} \cdots 
\binom{m - i_{\ell} - 1}{t_{\ell} - 2}.
\end{align*}
Iterative use of equation \eqref{5'} gives the same formula for $\B_{\mathbf{k}}(m)$:
$$
\B_{\mathbf{k}}(m) = \sum_{1 \le i_1 < \cdots < i_{\ell} < m}
i_1^{a_1} \dots i_{\ell}^{a_{\ell}} 
\binom{i_2 - i_1 - 1}{t_1 - 2} \cdots 
\binom{m - i_{\ell} - 1}{t_{\ell} - 2}.
$$
\end{proof}

{\it Remark.} A general case with the weight $(a_1, \ldots, a_{\ell}; t_1, \ldots, t_{\ell}, a)$ can be covered in a similar way. Combinatorial interpretation will be enriched by $\mathbf{k}$-partitions of $[n+1]$ having the property that the element $n+1$ is not a block or left-to-right minimum and $M = \max(a_1, \ldots, a_{\ell}, a)$. At the same time, the corresponding formulas 
$$B_{\mathbf{k}} = m^a B_{\mathbf{k} \setminus (k_{1}, \ldots, k_{n-a})} (m),\quad \B_{\mathbf{k}} = m^a \B_{\mathbf{k} \setminus (k_{1}, \ldots, k_{n-a})} (m)$$ hold for tuples $(k_1, \ldots, k_{n-a})$ with weight $(a_1, \ldots, a_{\ell}; t_1, \ldots, t_{\ell}, 0)$. 

\begin{crl}
If $\mathbf{k}$ has weight $(a_1, \ldots, a_{n-m}; t_1, \ldots, t_{n-m}; 0),$ then the following recurrence relations hold
$$
S_{(..., t_{n-m})}(n,m) =\sum_{i=0}^{t_{n-m} - 2}S_{(..., t_{n-m} - i)}(n-1,m-1)
+ m^{a_{n-m}} S_{(..., t_{n-m-1})}(n-1,m),
$$
$$
s_{(..., t_{n-m})}(n,m) =\sum_{i=0}^{t_{n-m} - 2}s_{(..., t_{n-m} - i)}(n-1,m-1)
+ (n-1)^{a_{n-m}} s_{(..., t_{n-m-1})}(n-1,m),
$$
or the following two types
$$
S_{(..., t_{n-m})}(n,m) = \sum_{i = 0}^{m} i^{a_{n-m}} \binom{t_{n-m} - 2 + m - i}{t_{n-m} - 2} S_{(..., t_{n-m-1})}(n-m-1+i,i),
$$
$$
s_{(..., t_{n-m})}(n,m) = \sum_{i = 0}^{m-1} i^{a_{n-m}} \binom{m - i - 1}{t_{n-m} - 2} s_{(..., t_{n-m-1})}(n-m-1+i,i).
$$
If we extend $m, n$ for all integers, then 
$$
\stirg{n}{m} = (-1)^{K} \ctirg{-m}{-n}.
$$
\end{crl}

If $k_1 = \cdots = k_n = 2,$ then we get the usual Stirling numbers:
$$B_{(2,\ldots, 2)}(m) = \stir{n + m}{m},\quad \B_{(2, \ldots, 2)}(m) = \ctir{m}{m - n}.$$

\begin{crl}
If $k_1 = \cdots = k_n = 3,$ then
$$B_{(3,\ldots, 3)}(m) = \stirr{n + m}{m},\quad \B_{(3, \ldots, 3)}(m) = \ctirr{m}{m - n},$$
where

\begin{enumerate}
\item $\stirr{n}{m}$ is a number of ordered pairs $(\pi_0, \pi_1)$ which satisfy the following properties:

\begin{enumerate}
\item $\pi_1$ is a partition of $[n]$ into $m$ blocks and $\pi_0$ is a segmented partition of a subset of $[n + 1]$ into $m$ blocks;

\item $\min(\pi_1) = \min(\pi_0)$;

\item if $x = \min([n] \setminus \min(\pi_1))$, then $x \sim_{\pi_0} 1$;
\end{enumerate}
\item $\ctirr{n}{m}$ is a number of ordered pairs $(\sigma_0, \sigma_1)$ which satisfy the following properties:
\begin{enumerate}
\item $\sigma_1$ is a permutation of $[n]$ having $m$ left-to-right minima and $\sigma_0$ is a segmented partition of $[n + 1]$ into $m$ blocks;

\item $\text{lmin}(\sigma_1) = \text{min}(\sigma_0)$;

\item if $x = \min([n] \setminus \text{lmin}(\sigma_1))$, then $x \sim_{\sigma_0} 1$.
\end{enumerate}
\end{enumerate}
\end{crl}

\begin{crl}
For $k_1 = \cdots = k_n = k > 2,$ we have 
$$B_{(k, \ldots, k)}(m) = \stirrr{n + m}{m}, \quad \B_{(k, \ldots, k)}(m) = \ctirrr{m}{m - n},$$
where
\begin{enumerate}
\item $\stirrr{n}{m}$ is a number of ordered pairs $(\pi_0, \pi_1)$ which satisfy the following properties:
\begin{enumerate}
\item $\pi_1$ is a partition of $[n]$ into $m$ blocks;

\item $\pi_0$ is a segmented partition of the multi-subset of $\{1^{k - 2}, \ldots, (n+1)^{k - 2} \}$ into $m$ (multiset) blocks so that any block contains one copy of its minimal element;

\item $\min(\pi_1) = \min(\pi_0)$;

\item if $x = \min([n] \setminus \min(\pi_1))$, then $x^{(k -2)} \sim_{\pi_0} 1$.
\end{enumerate}

\item $\ctirrr{n}{m}$ is a number of ordered pairs $(\sigma_0, \sigma_1)$ which satisfy the following properties:
\begin{enumerate}
\item $\sigma_1$ is a permutation of $[n]$ having $m$ left-to-right minima;

\item $\sigma_0$ is a segmented partition of the multiset $\{1^{k - 2}, \ldots, (n+1)^{k - 2} \}$ into $m$ multiset blocks so that any block contains one copy of its minimal element;

\item $\text{lmin}(\sigma_1) = \text{min}(\sigma_0)$;

\item if $x = \min([n] \setminus \text{lmin}(\sigma_1))$, then $x^{(k -2)} \sim_{\sigma_0} 1$.
\end{enumerate}
\end{enumerate}
\end{crl}

\begin{crl}\label{cat}
Suppose that $\mathbf{k}$ has weight $(a, \ldots, a; t, \ldots, t; 0)$ and length $n$. Then 
$$
B_{\mathbf{k}}(m) = S_{a,t}(n + m, m), \quad \B_{\mathbf{k}}(m) = s_{a,t}(m, m - n),
$$
where 

\begin{enumerate}
\item $S_{a,t}(n,m)$ is a number of ordered $(a+1)$-tuples $(\pi_0, \pi_1, \ldots, \pi_a)$ which satisfy the following properties:
\begin{enumerate}
\item $\pi_1, \ldots, \pi_a$ are partitions of $[n]$ into $m$ blocks;

\item $\pi_0$ is a segmented partition of the multi-subset of $\{1^{t - 2}, \ldots, (n+1)^{t - 2} \}$ into $m$ (multiset) blocks so that any block contains one copy of its minimal element;

\item $\min(\pi_1) = \min(\pi_0)$;

\item if $x = \min([n] \setminus \min(\pi_1))$, then $x^{(t -2)} \sim_{\pi_0} 1$.
\end{enumerate}

\item $s_{a,t}(n,m)$ is a number of ordered $(a+1)$-tuples $(\sigma_0, \sigma_1, \ldots, \sigma_a)$ which satisfy the following properties:
\begin{enumerate}
\item $\sigma_1, \ldots, \sigma_a$ are permutations of $[n]$ having $m$ left-to-right minima;

\item $\sigma_0$ is a segmented partition of the multiset $\{1^{t - 2}, \ldots, (n+1)^{t - 2} \}$ into $m$ multiset blocks so that any block contains one copy of its minimal element;

\item $\text{lmin}(\sigma_1) = \text{min}(\sigma_0)$;

\item if $x = \min([n] \setminus \text{lmin}(\sigma_1))$, then $x^{(t -2)} \sim_{\sigma_0} 1$.
\end{enumerate}
\end{enumerate}
\end{crl}

\section{Stirling numbers of odd type}\label{stirodd}

Let us define the \textit{Stirling numbers of odd type} of second and first kinds $S_{\text{odd}}(n,k)$ and $s_{\text{odd}}(n,k)$ by the following  recurrence relations
$$
S_{\text{odd}}(n, k) = S_{\text{odd}}(n-1, k-1) + k S_{\text{odd}}(n-1, k) + \delta_{n,k},
$$
$$
s_{\text{odd}}(n, k) = s_{\text{odd}}(n-1, k-1) + (n-1) s_{\text{odd}}(n-1, k) + \delta_{n,k},
$$
where $S_{\text{odd}}(0, 0) = s_{\text{odd}}(0, 0) = 0, S_{\text{odd}}(n, k) = s_{\text{odd}}(n, k)= 0$, if $n < k$;
$\delta_{n, k}$ is a Kronecker delta, $\delta_{n,k} = 1$ if $n = k$ and $\delta_{n,k} = 0$, otherwise. Note that $S_{\text{odd}}(n, n) = s_{\text{odd}}(n, n) = n$.

Now we introduce the notion of partitions and permutations with leader. Let 
 $\pi=\{B_1, \ldots, B_k\}$ be a partition of $[n]$ into $k$ blocks such that $\min(B_1) < \cdots < \min(B_k)$. 
Say that $(\ell,\pi)$ is {\it a partition with leader} $\ell$ if $\min(B_{\ell})=\ell.$ Similarly, for a cyclic presentation of a permutation $\sigma\in S_n$ 
as a product of cycles $\sigma=\sigma^{(1)}\cdots \sigma^{(k)},$ where $\min(\sigma^{(1)})<\cdots <\min (\sigma^{(k)}),$ 
say that $(\ell,\sigma)$ is {\it a permutation with leader}  $\ell,$ if $\min(\sigma^{(i)})=\ell.$ 
For example, for the partition $\pi=\{\{1,3\},\{2\},\{4\}\}$ the pairs $(1,\pi)$ and $(2,\pi)$ are partitions with leader.
Example for permutation: if $\sigma=14852763,$ then $\sigma=(1)(245)(38)(67)$ and $(1,\sigma), (2,\sigma),(3,\sigma)$ are permutations with leader.  

\begin{thm} 
The following properties hold for $S_{\rm{odd}}(n,k)$, $s_{\rm{odd}}(n,k)$:
\begin{itemize}

\item $S_{\rm{odd}}(n, k)$ is equal to the number of partitions with leader of $[n]$ into $k$ blocks.

\item $s_{\rm{odd}}(n, k)$ is equal to the number of permutations with leader of $[n]$ with $k$ cycles. 

\item If $k_0 = 1$ and $k_1 = \cdots = k_n = 2$, i.e., $\mathbf{n} = \{0^1, 1^2, \ldots, n^2 \}$, then 
$$B_{\mathbf{k}}(m) = S_{\rm{odd}}(n + m, m), \quad \B_{\mathbf{k}}(m) = s_{\rm{odd}}(m, m - n).$$ 
\item 
\begin{equation}\label{oddmul}
S_{\rm{odd}}(n+m, m) = \sum_{1 \le i_1 \le \cdots \le i_{n} \le m} i_1^2 i_2 \cdots i_{n},
\quad
s_{\rm{odd}}(m, m -n) = \sum_{1 \le i_1 < \cdots < i_{n} < m} i_1^2 i_2 \cdots i_{n}.
\end{equation}
\end{itemize}
\end{thm}

\begin{proof}
Let $S'(n,k)$ (resp. $s'(n,k)$) be the number of partitions (resp. permutations) with leader of $[n]$ into $k$ blocks (resp. cycles). 
If $n = k$, then there are $n$ partitions (resp. permutations) with leader and $S'(n, n) = s'(n, n) = n.$
Otherwise, if $n > k$, then $n$ is not a leader. 
If $n$ forms one separate block (resp. cycle), then there are $S'(n - 1, k - 1)$ (resp. $s'(n - 1, k - 1)$) partitions (resp. permutations) with leader.  
If $n$ belongs to other blocks (resp. cycles), there are $kS'(n -1, k)$ (resp. $(n-1)s'(n - 1, k)$) partitions (resp. permutations) with leader. Therefore, 
$$
S'(n,k) = S'(n-1,k-1) + kS'(n-1,k), \text{ if } n > k \text{ and } S'(n,n) = n,
$$
$$
s'(n,k) = s'(n-1,k-1) + (n-1)s'(n-1,k), \text{ if } n > k \text{ and } s'(n,n) = n.
$$
So,
$$
S_{\rm{odd}}(n,k) = S'(n,k), \quad s_{\rm{odd}}(n,k)= s'(n,k).
$$
If $n=0,$ then $B_{(0)}(m) = \B_{(0)}(m)= S_{\rm{odd}}(m, m) = s_{\rm{odd}}(m, m) = m.$ If $n \ge 1,$ then by the recurrence relations \eqref{main}, \eqref{main'}, 
$$B_{\mathbf{k}}(m) = B_{\mathbf{k}}(m-1) + mB_{\mathbf{k} \setminus 2}(m),$$ $$\B_{\mathbf{k}}(m) = \B_{\mathbf{k}}(m-1) + (m-1)\B_{\mathbf{k} \setminus 2}(m - 1),$$ which is the same as the recurrences $$S_{\rm{odd}}(n+m, m) = S_{\rm{odd}}(n+m - 1, m-1) + m S_{\rm{odd}}(n+m -1, m),$$ $$s_{\rm{odd}}(m, m - n) = s_{\rm{odd}}(m - 1, m-n-1) + (m-1) s_{\rm{odd}}(m -1, m - n).$$ Therefore,
$$
B_{\mathbf{k}}(m) = S_{\rm{odd}}(n + m, m), \quad \B_{\mathbf{k}}(m) = s_{\rm{odd}}(m, m - n).
$$ 

The formulas \eqref{oddmul} are consequences of Theorem \ref{spart} and \eqref{Sk}, \eqref{sk}.
\end{proof}

{\it Remark.} These combinatorial interpretations of $S_{\text{odd}}(n, k), s_{\text{odd}}(n,k)$ are a partial case of our general model with the weight $(2,1,\ldots, 1; 2, \ldots, 2; 0)$. In that case the $\mathbf{k}$-partition (resp. $\mathbf{k}$-permutation) system consists of two partitions (resp. permutations) $(\pi_1, \pi_2)$ (resp. $(\sigma_1, \sigma_2)$) differing only in one element that can be used as a leader. 

\

The usual Stirling numbers correspond to the case $k_1 = \cdots = k_n = 2$ and they can be defined as \textit{Stirling numbers of even type}. The numbers $S_{\text{odd}}(n, k)$ in  Knuth \cite{knuth-opt} denoted as half-integer Stirling numbers.

We may also note that the presented combinatorial meanings of $S_{\text{odd}}(n, k), s_{\text{odd}}(n,k)$ clearly imply their connection with the $r$-Stirling numbers introduced by Broder \cite{broder}. Namely, 
$$
S_{\text{odd}}(n, k) = \sstir{n}{k}_1 + \sstir{n}{k}_2 + \cdots + \sstir{n}{k}_k, \quad s_{\text{odd}}(n, k) = \cstir{n}{k}_1 + \cstir{n}{k}_2 + \cdots + \cstir{n}{k}_k,
$$
where $\cstir{n}{k}_r, \sstir{n}{k}_r$ are $r$-Stirling numbers of the first and second kinds, which count the number of permutations (resp. partitions) of $[n]$ with $k$ cycles (resp. blocks) such that the numbers $1, \ldots, r$ are in distinct cycles (resp. blocks). 

\section{Generalization of central factorial numbers} \label{gencent}
Let  $S_t(n,k),s_t(n,k)$ be the numbers  defined by the following relations 
\begin{align*}
x^n = \sum_{k = 0}^{n} S_t(n, k) \prod_{i = 0}^{k-1} (x-i^t),
\text{\qquad }
\prod_{i = 0}^{n-1} (x+i^t) = \sum_{k = 0}^{n} s_t(n, k) x^k.
\end{align*}
They satisfy the following  recurrence relations
\begin{align*}
S_t(n, k) &= S_t(n-1, k-1) + k^t S_t(n-1,k), \quad S_t(0, 0) = 1, S_t(n,k) = 0 \text{ if } n < k;\\
s_t(n, k) &= s_t(n-1, k-1) + (n-1)^t s_t(n-1,k), \quad s_t(0, 0) = 1, s_t(n,k) = 0 \text{ if } n < k.\end{align*}

If $t = 1$, then $S_1(n, k), s_1(n, k)$ became the usual Stirling numbers of the second and first kind, respectively. If $t = 2$, then $S_2(n, k), s_2(n, k)$ refer to the central factorial numbers $T(2n, 2k), t(2n,2k)$ (Riordan \cite{riordan}) defined by
$$
x^n = \sum_{k = 0}^n T(n,k) x \prod_{i = 1}^{k - 1} (x + k/2 - i), \quad x \prod_{i = 1}^{n-1} (x + n/2 - i) = \sum_{k = 0}^n t(n,k) x^k.
$$
The numbers $S_t(n,k), s_t(n,k)$ are partial cases of $\stirg{n}{k}, \ctirg{n}{k}$ defined in Section \ref{str}, if ${\bf k}$ consists of several repetitions of $t$-series with ends $2.$   

Denote by $\mathbf{t}n $ the $nt$-tuple defined by 
$$
\mathbf{t}n = (k_1, \ldots, k_{nt}) = (\underbrace{\overbrace{1,1,\dots, 1,2}^{t \text{ numbers }}, \overbrace{1,1,\dots, 1,2}^{t \text{ numbers }}, \dots, \overbrace{1,1,\dots, 1,2}^{t \text{ numbers }}}_{n \text{ blocks by $t$ numbers}}).
$$
The multiset of type ${\mathbf{t}}n$  has length $n$ and weight $(t, \ldots, t; 2, \ldots, 2; 0)$.
For example, if $t = 3, n = 2,$ then $(k_1, \ldots, k_6) = (1, 1, 2, 1, 1, 2)$ and the corresponding multiset is $\mathbf{nt} = \{1^{k_1}, \ldots, 6^{k_6} \} = \{1, 2, 3, 3, 4, 5, 6, 6 \}.$ 

The $\mathbf{t}n$-Stirling poset $P_{{\bf t}n}$ 

\noindent\begin{center}
\begin{tabular}{c}
\includegraphics{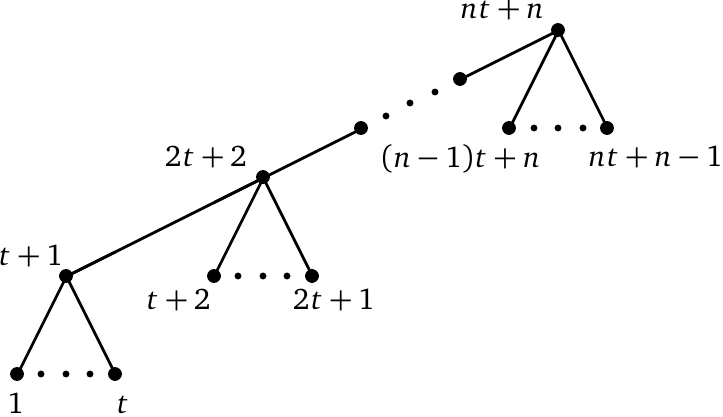}
\end{tabular}

{\bf Fig. 4.} The $\mathbf{t}n$-Stirling poset $P_{{\bf t}n}$.
\end{center}

satisfies the following relations
$$
\Omega(P_{{\bf t}n}, m) = S_t(n + m, m), \quad \overline\Omega(P_{{\bf t}n}, m) = s_t(m, m - n).
$$

\begin{thm} The following properties hold for $S_t(n,k), s_t(n,k)$:
\begin{itemize}
\item 
$S_t(n,k)$ is a number of ordered $t$-tuples $(\pi_1, \ldots,\pi_t),$ where $\pi_1, \ldots,\pi_t$ are partitions of $[n]$ into $k$ blocks such that  $\min(\pi_1) = \cdots = \min(\pi_t)$.

\item 
$s_t(n,k)$ is a number of ordered $t$-tuples $(\sigma_1, \ldots,\sigma_t),$ where $\sigma_1,\ldots,\sigma_t$  are permutations of $[n]$ that have $k$ cycles, such that  $\min(\sigma_1) = \cdots = \min(\sigma_t).$ 

\item $B_{\mathbf{t}n}(m) = S_{t}(n+m,m)$ and $\B_{\mathbf{t}n}(m) = s_{t}(m,m - n).$
\item 
\begin{equation}\label{tmul}
S_t(n + k, k) = \sum_{1 \le i_1 \le \cdots \le i_{k} \le n} i_1^t \cdots i_k^t, \qquad
s_t(n, n - k) = \sum_{1 \le i_1 < \cdots < i_{k} < n} i_1^t \cdots i_k^t.
\end{equation}
\end{itemize}
\end{thm}

\begin{proof} Let $S'_t(n,k)$ (resp. $s'_t(n,k)$) be the number of ordered $t$-tuples of partitions (resp. permutations) of $[n]$ into $k$ blocks (resp. cycles) having the same set of blocks (resp. cycles) minima.

If we have a separate block $\{ n\}$ (resp. cycle $(n)$), then $n$ is a minimal element and this block should appear in every partition $\pi_1, \ldots, \pi_t$ (resp. permutation $\sigma_1, \ldots, \sigma_t$), and the number of ways to form these tuples is $S'_t(n - 1, k - 1)$ (resp. $s'_t(n - 1, k - 1)$). Otherwise, if $n$ belongs to blocks (resp. cycles) with the other elements, then for any partition $\pi_i$ $(1 \le i \le t)$ (resp. permutation $\sigma_i$), there are $k$ (resp. $n-1$) ways to put $n$ in $k$ blocks (resp. cycles) of $\pi_i$ (resp. $\sigma_i$). So, there are totally $k^t S'_t(n-1, k)$ (resp. $(n-1)^t s'_t(n-1, k)$) ways to form $\pi_1, \ldots, \pi_t$ (resp. $\sigma_1, \ldots, \sigma_t$). $S'_t(n,k)$ (resp. $s'_t(n,k)$) satisfies the same recurrence as $S_t(n,k)$ (resp. $s_t(n,k)$) and $S'_t(1,1) = S_t(1,1), s'_t(1,1) = s_t(1,1)$. Therefore, $$S_t(n,k) = S'_t(n,k), \quad s_t(n,k) = s'_t(n,k).$$

If $n = 0$, then $B_{\emptyset}(m) =\B_{\emptyset}(m) = S_t(m, m) = s_t(m,m)= 1.$ By the recurrence relations \eqref{main}, \eqref{main'}, $$B_{\mathbf{t}n}(m) = B_{\mathbf{t}n}(m - 1) + m^t B_{\mathbf{t}(n-1)}(m), \quad \B_{\mathbf{t}n}(m) = \B_{\mathbf{t}n}(m - 1) + (m-1)^t \B_{\mathbf{t}(n-1)}(m).$$ It is easy to see that $S_{t}(n+m,m), s_{t}(m,m - n)$ satisfy the same recurrence relations and the corresponding initial values coincide. 

The formulas \eqref{tmul} imply from Theorem \ref{spart} and \eqref{Sk}, \eqref{sk}.
\end{proof}

Note that the described combinatorial meanings of $S_t(n,k), s_t(n,k)$ can be refined from Corollary \ref{cat} and in case of $t = 2$ they are similar to Dumont's interpretations of the central factorial numbers (see \cite{dumont} and \cite{foata, zeng}). Other combinatorial interpretations of the generalized central factorial numbers $S_t(n,k)$ have previously been considered in \cite{domaratzki} and in a more general version in \cite{han}.

\end{document}